\documentclass[final,1p,times]{elsarticle}

\biboptions{numbers,sort&compress}
\usepackage[colorlinks=true,linkcolor=blue]{hyperref}%
\usepackage{amsmath}
\usepackage{amsthm} 
\usepackage{xcolor}
\usepackage{amssymb}
\newtheorem{theorem}{Theorem}[section]
\newtheorem{lemma}[theorem]{Lemma}

\newtheorem{example}{Example}
\usepackage{subcaption}
\date{}
\begin{document}

\begin{frontmatter}
	
	\title{Effort Dynamics with Selective Predator Harvesting of a Ratio-dependent Predator-Prey System}


\author[mymainaddress]{Uganta Yadav\corref{mycorrespondingauthor}}
\cortext[mycorrespondingauthor]{Corresponding author}
\ead{uyadav@ma.iitr.ac.in}
\author[mysecondaryaddress]{Sunita Gakkhar}
	
	\address[mymainaddress]{Department of Mathematics, IIT Roorkee, 247667, India.}
	\address[mysecondaryaddress]{Department of Mathematics, IIT Roorkee, 247667, India.}
	
	\begin{abstract}
		Sustainable harvesting of renewable resources is a core issue in fisheries and forest management. The over-exploitation of valuable resources has led to undesirable extinction of several biological species causing non-reversible damage to ecology and environment. In this paper, the selective predator harvesting is considered in a predator prey system. The logistic growth of prey and ratio-dependent functional response is assumed. The dynamically varying effort is employed for selective harvesting of predator species using Michaelis-Menten type function The boundedness, positivity and persistence of the nonlinear dynamical system is established. The system being non-singular about origin, its dynamics is explored by transforming it to a regular system. It is observed that the system may collapse even with positive initial conditions attracted to one of the boundary points under certain conditions. The local stability of various other equilibrium states of the dynamic model is investigated. The region of attraction of interior equilibrium state is obtained using Lyapunov stability. The occurrence of hopf bifurcation about the coexistence equilibrium point with respect to the parameter m (fraction of predators available for harvesting) is established. The sustainable harvesting is possible in two ways firstly in the form of stable interior equilibrium state and secondly in the form of limit cycle. Considering the fraction m as a control variable, the optimal harvesting policy is obtained using Pontraygin's maximum principle. The two-parameter bifurcation diagram is obtained with respect to c and d (the cost of harvesting and the death rate of predators respectively). The two bi-stability regions of parameter space are identified. One of the regions shows bi-stability of the origin and the effort free equilibrium state. The other region involves the bi-stability of the origin and interior equilibrium state. It is proved that for sustainable harvesting of predators the ratio of p to c (price to cost) should be greater than a certain threshold value dependent on predator density. The analytical results are illustrated numerically with different parametric values.
	\end{abstract}
	
	
\end{frontmatter}

\section{Introduction}
The vital relationship between predators and prey has been dominant subject in mathematical ecology because of its global existence, significance and for understanding of interacting populations in the natural environment. The functional response in prey predator model plays a significant role in the dynamics of ecological models. Holling justified the functional forms of types I, II and III using 
	a straightforward argument based on the division of an individual 
	predator’s time into two periods: ‘searching for’ and ‘handling of’ 
	prey  \cite{dawes2013derivation}. However, Arditi and Ginzburg argued that the functional response should be a function of ratio of prey biomass to predator biomass and this argument was supported by field observations and laboratory experiments \cite{arditi1989coupling}. Leslie Gower model was the first ratio-dependent model \cite{gupta2012bifurcation,gupta2013bifurcation}. The ratio dependent models have richer dynamics. It was observed that both the paradox of enrichment and the paradox of biological control are not valid for ratio-dependent systems\cite{alebraheem2018relationship}.\\
Considering $x(t),y(t)$ as prey and predator densities at time $t$ and taking all the parameters to be positive, a ratio-dependent predator prey system has the form \cite{xiao2006dynamics}
\begin{equation}
	\begin{split}\label{eq:(1)}
		\frac{dx}{dt}=x(1-x)-\frac{axy}{x+y}\\
		\frac{dy}{dt}=\frac{exy}{x+y}-dy
	\end{split}
\end{equation}
This is based on the assumptions that the prey population grows logistically in the absence of predators. The predators die in the absence of prey at the rate $d$ whereas $a,e$ are positive constants.  A prominent attribute of a ratio-dependent model is its immense dynamics near origin\cite{jost1999deterministic,flores2014dynamics}. Xiao and Ruan had shown the existence of complex dynamics of System \ref{eq:(1)} and the detailed investigation on the periodic solutions and its diverse bifurcations have been carried out\cite{xiao2001global}. \\
The growing need for more food and resources has led to an increased exploitation of several biological resources. On the other hand there is a global concern to protect the ecosystem at large. Although it is beneficial to mankind but unplanned, over indulgence in harvesting may lead to extinction of the harvesting species. Therefore, the sustainable harvesting policies are required to maintain a balance between economic progress and ecological diversity. The extensive techniques for optimal management of renewable resources and the long  term benefits were presented by Clark\cite{clark1976mathematical}.  May et.al. proposed two types of harvesting $(i)$ independent of the population biomass. $(ii)$ directly proportional to the population biomass. Clark\cite{clark1976mathematical} proposed non-linear harvesting which is more realistic from biological and economical perspective. The non-linear harvesting term of Michaelis Menten type is $h(E,x)=\frac{qEx}{m_1E+m_2x}$ where $q$ is the catchability coefficient, $E$ is the effort on harvesting and $m_1,m_2$ are positive constants. Here, $\lim_{E \to \infty} h(E,x)=\frac{qx}{m_1}$, which is the form of proportional harvesting. Similarly $h(E,x)$ tends to constant harvesting as $lim_{x \to \infty} h(E,x)=\frac{qE}{m_2}.$ This harvesting function makes sure that the yield of harvesting predator remains bounded whenever effort tends to infinity.\\ Considering ratio-dependent functional response, Kar et.al.\cite{kar2006bioeconomic} examined an optimal  policy for combined harvesting of two species harvested at a rate proportional to both stock and effort. It was established that the persistence of the system depends on  fishing effort. Lenzini considered different non-constant predator harvesting functions and studied emergence of different bifurcations\cite{kar2010effort}. Hopf bifurcation and transcritical bifurcation may occur with certain parametric conditions in case of linearly varying harvesting function. However, pitchfork bifurcation was observed with rational harvesting function. 
Stabilizing effect of marine reserves on fishery dynamics  \\
Dongpo Hu\cite{hu2017stability} investigated the Leslie-Gower predator prey model with Michaelis Menten type predator harvesting. The harvesting influences the system to have complex dynamical behavior including the saddle–node, transcritical, Hopf and Bogdanov–Takens bifurcations. Existence of such bifurcations stipulates that the over exploitation of resources may lead to extinction of the species. However, modified Leslie-Gower with Michaelis Menten type prey harvesting is studied by Gupta et.al.\cite{gupta2013bifurcation}. It is shown that with certain parametric conditions a situation occurs where the solutions are dependent on the
initial condition i.e., the solutions converge to the prey extinction equilibrium state for some initial
values in one region and they converge to the coexistence equilibrium state if the initial conditions lie in the other region. Das et al.\cite{das2009bioeconomic}  examined an optimal harvesting policy using Pontragin’s maximum principle for combined Michaelis Menten type harvesting which grow logistically and predation is of Holling type II. \\
It is customary to assume constant effort in harvesting models. But including the effort dynamics in the system incorporates the economic interest along with the ecological benefits to maintain sustainable development of the species\cite{kar2010effort}. Our study is focused on the ratio-dependent predator prey harvesting system incorporating effort dynamics. In this paper, it is assumed that prey are of no economic value and predator are harvested with non-linear harvesting. The predator harvesting indirectly affects prey due to reduction in predation pressure. Effort dynamics is incorporated in the system to have a better insight from the bionomic perspective.\\
This paper is organized as follows: Section 2 includes the model and mathematical preliminaries. Section 3 elaborates on the equilibrium states and their stability. In the next section 4, Hopf bifurcation analysis is implemented followed by global stability in section 5. Later the optimal harvesting policy is studied in section 6. Lastly in section 7 the numerical simulation is carried out to give the better understanding of analytical results and validating them. 
\section{Model and Mathematical Preliminaries:}
Considering $x(t)$ as prey, $y(t)$ as predator and $E(t)$ as harvesting effort, the mathematical model with ratio-dependent functional response is written as
\begin{equation}\label{eq:(2)}
	\begin{split}
		\frac{dx}{dt}=x(1-x)-\frac{axy}{x+y}=f_1(x,y,E)\\
		\frac{dy}{dt}=\frac{exy}{x+y}-dy-\frac{qmEy}{m_1E+m_2y}=f_2(x,y,E)\\
		\frac{dE}{dt}=\rho(\frac{pqmyE}{m_1E+m_2y}-cE)=f_3(x,y,E)
	\end{split}
\end{equation}
The prey species is growing logistically. The ratio-dependent functional response is considered in the model. For effort dynamics the Michael-Menten type nonlinear function is considered. The model is associated with initial conditions 
\begin{equation}\label{eq:(3)}
	 x(0)\geq 0,y(0)\geq 0,E(0)\geq 0.
 \end{equation}
The model parameters $a,e \hspace{0.1cm}\mbox{and}\hspace{0.1cm} d$ are positive and have usual meaning as in a prey- predator model. The parameters $m,p,c \hspace{0.1cm} \mbox{and}\hspace{0.1cm} \rho$ are fraction of predators available for harvesting, constant selling price, harvesting cost per unit effort and the stiffness parameter respectively. 
All the functions $f_{i},(i=1,2,3)$ and their partial derivatives are continuous in $\mathbb{R}^3_+.$ In consequence, they are Lipschitzian in $\mathbb{R}^3_+$. Hence, the solution of system \ref{eq:(2)} with non-negative initial conditions exists and is unique. 
\begin{lemma}{(Positive Invariance)}
	The system \ref{eq:(2)} is positively invariant in  $\mathbb{R}^3_+.$
\begin{proof}
Considering $X=(x,y,E)^T\in\mathbb{R}^3$ and $F:\mathbb{C_+}\longrightarrow \mathbb{R}^3$, $F \in \mathbb{C^\infty}(\mathbb{R}^3)$ the system \ref{eq:(2)}
 is written in the matrix form as $\dot{X}=F(X)=(f_1,f_2,f_3)^T$ with $X(0)\geq 0$.
It is observed that $F_i(X)|_{X_i=0}\geq 0$ whenever $X(0)\in \mathbb{R}_+^3$ such that $X_i=0$. Accordingly, for given positive initial conditions, the solution remains non-negative for $t>0$.
\end{proof}
\end{lemma}

\begin{lemma}{(Boundedness)}
	The system \ref{eq:(2)} admits bounded solution in the domain \\ $D=\{(x,y,E)\in \mathbb{R}^3:0\leq x+\frac{a}{e}y+\frac{a}{\rho ep}E\leq \frac{M}{\mu}\}$.

\begin{proof}
Let us define
\begin{equation}\label{eq:(4)}
	w(t)=x(t)+\frac{a}{e}y(t)+\frac{aE(t)}{\rho ep}
\end{equation} 
Using \ref{eq:(2)}, the time derivative of $w$ is computed as 
$$\frac{dw}{dt}=x-x^2+\frac{ady}{e}-\frac{caE}{pe}$$
For arbitrary $\mu > 0$,
$$\frac{dw}{dt}+\mu w \leq-x^2+(1+\mu)x+(\mu-d)\frac{a}{e}y+\frac{aE}{ep}(\frac{\mu}{\lambda}-c)$$
Choosing $\mu<min \{d,c\lambda\}$, there exists $M=\frac{(1+\mu)^2}{4}$ such that 
$$\frac{dw}{dt}+\mu w \leq M$$
Use of theory of differential inequality gives
$$ 0\leq w(t)\leq \frac{M}{\mu}(1-e^{-\mu t})+w(x(0),y(0),E(0))e^{-\mu t} $$
or  $0\leq w(t) \leq \frac{M}{\mu}$, as $t \longrightarrow \infty$.
\end{proof}
\end{lemma}
As a consequence of limited resources, there is natural hindrance to the growth of species. So, boundedness of solutions is an imperative condition for the system to be biologically rational.

\begin{lemma}{(Persistance)}\label{lemma 2.3}
The system \ref{eq:(2)} is persistant if
\begin{equation}\label{eq:5}	e>d\hspace{0.5cm}\mbox{and}\hspace{0.5cm}\frac{p}{c}>\frac{m_2}{qm}.
\end{equation} 
\begin{proof} Since variables $x,y,E$ are positively invariant, so from the prey equation
$$\frac{dx}{dt}\leq x(1-x)$$ 
Now, the Comparison test gives 
$$\limsup_{n\to \infty}x(t)\leq 1$$
Similarly for predator equation
$$\hspace{-2.5cm}\frac{dy}{dt}=\frac{exy}{x+y}-dy-\frac{qmEy}{m_1E+m_2y}$$
$$\hspace{-3cm}\leq \frac{y}{1+y}(e-d-dy)$$
$$\implies \limsup_{n\to\infty}y(t)\leq \frac{-d+e}{d}=\frac{e}{d}-1=\bar{y}$$ \\
So, $\limsup_{n\to\infty}y(t)>0$ whenever $e>d.$\\
From the third equation of the system
$$\frac{dE}{dt}=\rho(\frac{pqmyE}{m_1E+m_2y}-cE)$$
$$\hspace{3.5cm}\leq\frac{\rho E}{m_1E+m_2\bar{y}}((pqm-cm_2)\bar{y}-cm_1E)$$
which implies $$\limsup_{n\to\infty} E(t)\leq \frac{(pqm-cm_2)}{cm_1}\bar{y}\hspace{0.8cm}$$\
which is positive provided $\frac{p}{c}>\frac{m_2}{qm}.$
\end{proof}
\end{lemma}
Therefore, for sustainable harvesting of predators the ratio of price to cost should be greater than a certain threshold value dependent on predators. Otherwise  harvesting will not be profitable.
\section{Equilibrium states and their stability:}
The system \ref{eq:(2)} has the following three boundary equilibrium states and one interior or bionomic equilibrium state:\\
\textbf{Boundary equilibrium:}
\begin{enumerate}
	\item The trivial state $E_0=(0,0,0)$ exists for all parametric values.
	\item The predator and effort free axial equilibrium $E_1=(1,0,0)$ always exist.
	\item Let \begin{equation}\label{eq:(6)}
		x_1=1-\frac{a(e-d)}{e},y_1=\frac{(e-d)}{d}x_1.
	\end{equation}
Then the effort free equilibrium $E_2=(x_1,y_1,0)$ 
	exists provided the following conditions are satisfied:
	
	\begin{equation}\label{eq:(7)}
		\{  e>d :0<a\leq 1\}\hspace{0.2cm}\mbox{and}\hspace{0.2cm} \{d<e<\frac{ad}{a-1}:a>1\}.
	\end{equation} 
\end{enumerate}  
\textbf{Bionomic equilibrium:}\\
Let \begin{equation}\label{eq:(8)}
	x^{*}=1-a+\frac{ad}{e}+\frac{a(pqm-cm_2)}{epm_1}, y^{*}=\frac{x^{*}(1-x^{*})}{a+x^{*}-1},E^{*}=\frac{pqm-cm_2}{cm_1}y^{*}.
\end{equation} 
Then the interior equilibrium point $E_3=(x^{*},y^{*},E^{*})$ 
exists if the following condition is satisfied:
\begin{equation}\label{eq:(9)}
	0<pqm-cm_2<pm_1(e-d).
\end{equation}

\subsection{Stability of Origin}

As the system is singular at $E_0$, its stability is investigated using blow up technique\cite{hsu2003ratio}. Accordingly, the system \ref{eq:(2)} is transformed using the transformation $(u,y,v)=(\frac{x}{y},y,\frac{y}{E})$ as
\begin{equation}\label{eq:(10)}
	\begin{split}
		\frac{du}{dt}=\frac{u}{1+u}(-A-Bu)-u^2y+\frac{mqu}{m_1+m_2v}\\
		\frac{dy}{dt}=-dy+\frac{euy}{u+1}-\frac{mqy}{m_1+m_2v}\\
		\frac{dv}{dt}=-dv+\frac{euv}{1+u}-\frac{mqv}{m_1+m_2v}-\rho(v(\frac{pmqv}{m_1+m_2v}-c))
	\end{split}
\end{equation}
The new model constants are defined as: $A=a-d-1, B=e-d-1.$\\
For the transformed system \ref{eq:(10)}, the interest is only in the boundary equilibrium points:
\begin{enumerate}
	\item The state $E_{00}=(0,0,0)$ exists for all parametric values.
	\item Let $\bar{x}=\frac{A-\frac{mq}{m_1}}{\frac{mq}{m_1}-B}$, then the state $E_{10}=(\bar{x},0,0)$ exists if one of the following conditions is satisfied:
	\begin{subequations}
		\begin{align}
			\{A<\frac{mq}{m_1}<B\} \label{eq:(11a)} \\
			or \hspace{0.1cm} \{B<\frac{mq}{m_1}<A\}.\label{eq:(11b)}
		\end{align}
	\end{subequations}  
	\item Let $\bar{E}=\frac{m_1(\rho c-d)-mq}{\rho pmq-m_2(\rho c-d)}$. Then the state $E_{100}=(0,0,\bar{E})$ exists provided one of the following conditions is satisfied:
	\begin{subequations}
		\begin{align}
			\{0<\frac{\rho pmq}{m_2}<\rho c-d<\frac{mq}{m_1}\} \label{eq:(12a)}\\or \hspace{0.1cm} \{0<\frac{mq}{m_1}<\rho c -d<\frac{\rho pmq}{m_2}\}\label{eq:(12b)}
		\end{align}
	\end{subequations}
	\end{enumerate}
The Jacobian matrix $J_0=[a_{ij}]_{3\times 3}$ of the system \ref{eq:(10)} computed at a point $(u,y,E)$ is given below\\
$a_{11}=-\frac{A+2Bu+Bu^2}{(1+u)^2}-2uy+\frac{mq}{(m_1+m_2v)},a_{12}=-u^2,a_{13}=\frac{-mqum_2}{(m_1+m_1v)^2}$,\\
$a_{21}=\frac{ey}{(1+u)^2},
a_{22}=-d+\frac{eu}{1+u}-\frac{mq}{m_1+m_2v},a_{23}=\frac{mqym_2}{(m_1+m_2v)^2}$\\
$a_{31}=\frac{ev}{(1+u)^2},a_{32}=0,a_{33}=-d+\frac{eu}{1+u}-\frac{m_1mq}{(m_1+m_2v)^2}-\frac{\lambda pmq(2vm_1+v^2m_2)}{(m_1+m_2v)^2}+\rho c.$

\begin{lemma}
The state $E_{00}$ of the transformed system \ref{eq:(10)} is stable whenever the condition given below is satisfied:
\begin{equation}\label{eq:(13)}
	\rho c-d<\frac{mq}{m_1}<A
\end{equation}
\begin{proof}
	The matrix $J_0$  at point $E_{00}$ is a diagonal matrix with diagonal entries 
	$$a_{11}=-A+\frac{mq}{m_1},\hspace{0.2cm}a_{22}=-d-\frac{mq}{m_1},\hspace{0.2cm} a_{33}=\rho c-d-\frac{mq}{m_1}$$
	Accordingly, $E_{00} $ is stable provided condition \ref{eq:(13)} is satisfied.
\end{proof}
\end{lemma}
\begin{lemma}

The state $E_{10}$ of the transformed system \ref{eq:(10)} is stable whenever it exists under condition \ref{eq:(11a)} and the following condition is satisfied:
\begin{equation}\label{eq:(14)}
	e<\frac{(1+\bar{x})}{\bar{x}}(d+\frac{mq}{m_1}-\rho c).
\end{equation}
\begin{proof}
	The matrix $J_0$ at $E_{10}$ is a upper triangular matrix with 
	$$a_{11}=\frac{(A-B)\bar{x}}{(1+\bar{x})^2},\hspace{0.2cm} a_{22}=-\frac{e\bar{x}}{1+\bar{x}}-(d+\frac{mq}{m_1}),\hspace{0.2cm} a_{33}=\frac{e\bar{x}}{1+\bar{x}}-(d+\frac{mq}{m_1}+\rho c)$$
	For stability, all the eigenvalues of $J_{0}$ at $E_{10}$ must be negative. Now, $a_{11}<0$ is ensured under existence condition \ref{eq:(11a)} and $a_{22}<a_{33}$. So, $a_{33}<0$ under condition \ref{eq:(14)} . Hence proved.
\end{proof}
\end{lemma}
Therefore, it is concluded that the state $E_{10}$ is unstable if it exists under \ref{eq:(11b)} or violates \ref{eq:(14)}.
\begin{lemma}

The state $E_{100}$ of the transformed system \ref{eq:(10)} is stable whenever it exists under condition \ref{eq:(12b)} and the following condition is satisfied:
\begin{equation}\label{eq:(15)}
	\frac{mq}{m_1+m_2\bar{E}}<A.
\end{equation}
\begin{proof}
	At $E_{100}$ the matrix $J_0$ is a lower triangular matrix with 
	$$a_{11}=-A+\frac{mq}{m_1+m_2\bar{E}},\hspace{0.2cm}
	a_{22}=-d-\frac{mq}{m_1+m_2\bar{E}},\hspace{0.2cm}
	a_{33}=\frac{mqv}{(m_1+m_2\bar{E})^2}(m_2-\rho p m_1)$$
	Accordingly, $E_{100}$ is stable whenever all the eigenvalues of $J_{0}$ at $E_{100}$ are negative. Now, $a_{33}<0$ under existence condition \ref{eq:(12b)} and $a_{11}<0$ under condition \ref{eq:(15)}. Hence, proved
\end{proof}
\end{lemma}
It is concluded that the state $E_{100}$ is unstable if it exists under \ref{eq:(12a)} or violates \ref{eq:(15)}.

\begin{theorem}{}\label{thm:3.4}
	The stability of $E_0$ is possible if one of the following is true
	
	\renewcommand{\theenumi}{(\roman{enumi})}%
	\begin{enumerate}
		\item The state $E_{00}$ of system \ref{eq:(10)} is stable.
		\item The state $E_{10}$ of system \ref{eq:(10)} exists and is stable.
		\item The state $E_{100}$ of system \ref{eq:(10)} exists and is stable.
	\end{enumerate}
\end{theorem}
\begin{proof}

	It can be noticed that
	\begin{itemize}
		\item $(u,y,v) \rightarrow E_{00} $  if and only if $(x,y,E)\rightarrow (0,0,0)$ when $x(t)\rightarrow 0$ faster than $y(t)$ and $y(t) \rightarrow 0$ faster than $z(t)$.
		\item $(u,y,v) \rightarrow E_{10} $ if and only if $(x,y,E)\rightarrow (0,0,0)$ when $y(t)\rightarrow 0$ faster than $z(t)$ and $x(t)\rightarrow 0$ at finite rate as $y(t)\rightarrow 0$.
		\item $(u,y,v) \rightarrow E_{100} $  if and only if $(x,y,E)\rightarrow (0,0,0)$ when $x(t)\rightarrow 0$ faster than $y(t)$ and $z(t)\rightarrow 0$ at finite rate as $y(t)\rightarrow 0$.

	\end{itemize}
\end{proof}
\textbf{\underline{Note:}}	
It is noted from the existence and stability condition of $E_{00},E_{10}$ and $E_{100}$ that one or more of the states may exist simultaneously but only one will be stable at a time.
\subsection{Stability of }
To investigate the stability of predator-effort free equilibrium $E_1$, the following transformed system $(x,y,w)$ using the new variable $w=\frac{y}{E}$ is studied:
\begin{equation}\label{eq:(16)}
	\begin{split}
		\frac{dx}{dt}=x(1-x)-\frac{axy}{x+y}\\
		\frac{dy}{dt}=-dy+\frac{exy}{x+y}-\frac{mqy}{m_1+m_2w}\\
		\frac{dw}{dt}=-dw+\frac{exw}{x+y}-\frac{mqw}{m_1+m_2w}-\rho w(\frac{pmqw}{m_1+m_2w}-c)
	\end{split}
\end{equation}

Two axial/boundary steady states of the transformed system are 
\begin{enumerate}
	\item The state $E_{01}=(1,0,0)$ always exist.
	\item The state $E_{001}=(1,0,w^*)$ exists provided 
	\begin{subequations}
		\begin{align}
			\{0<\frac{mq}{m_1}<	e+\rho c-d< \frac{\rho pmq}{m_2} \} \label{eq:(17a)}\\ or \hspace{0.5cm} \{0<\frac{\rho pmq}{m_2}<e+\rho c-d<\frac{mq}{m_1}\}\label{eq:(17b)}
		\end{align}
	\end{subequations}
	
	where $w^*$ is defined as follows:
	\begin{equation}\label{eq:(18)}
		w^*=\frac{(-d+e+\rho c)m_1-mq}{\rho pmq-m_2(-d+e+\rho c)}.
	\end{equation}
	
\end{enumerate}
The Jacobian matrix $J_1=[b_{ij}]_{3\times 3}$ of the transformed system \ref{eq:(16)} is\\
$b_{11}=1-2x-\frac{ay^2}{(x+y)^2},\hspace{0.2cm}b_{12}=\frac{ax^2}{(x+y)^2},\hspace{0.2cm}b_{13}=0$\\
$b_{21}=\frac{ey^2}{(x+y)^2},\hspace{0.2cm}b_{22}=-d+\frac{ex^2}{(x+y)^2}-\frac{mq}{(m_1+m_2w)},\hspace{0.2cm}b_{23}=\frac{mqm_2y}{(m_1+m_2w)^2}$\\
$b_{31}=\frac{exy}{(x+y)^2},\hspace{0.2cm}b_{32}=-\frac{exw}{(x+y)^2},\hspace{0.2cm}b_{33}=-d+\frac{ex}{x+y}-\frac{m_1mq}{(m_1+m_2w)^2}-\rho pqm\frac{2m_1w+m_2w^2}{(m_1+m_2w)^2}+\rho c$\\
\begin{lemma}
The state $E_{01}$ is stable provided the following condition is satisfied:
\begin{equation}\label{eq:(19)}
	e+\rho c-d <\frac{mq}{m_1}.
\end{equation}
\begin{proof}
	The matrix $J_1$ at $E_{01}$ is a upper triangular matrix matrix with
	$$b_{11}=-1, \hspace{0.2cm} b_{22}=e-d-\frac{mq}{m_1},\hspace{0.2cm}b_{33}=e+\rho c-d-\frac{mq}{m_1}$$
	Accordingly, the state $E_{01}$ is stable under condition \ref{eq:(19)}.
\end{proof}
\end{lemma}
\begin{lemma}
Let the state $E_{001}$ exists under condition \ref{eq:(17a)}. Then it will be stable provided it satisfies the following condition:
\begin{equation}\label{eq:(20)}
	e<d+\frac{mq}{m_1+m_2w^*}.
\end{equation}
\begin{proof}
	The matrix $J_1 $ is simplified to the following form at $E_{001}$ is 
	$$J_1|_{E_{001}}=\begin{bmatrix}
		-1 && -a && 0\\
		0&&e-d-\frac{mq}{m_1+m_2w^*}&&0\\
		0&&-ew^*&&\frac{mqw^*}{(m_1+m_2{w^*})^2}(m_2-\rho p m_1)
	\end{bmatrix}$$
	Its eigenvalues are 
	$$\lambda_1=-1\hspace{0.2cm},\lambda_2=e-d-\frac{mq}{(m_1+m_2w^*)},\hspace{0.2cm}\lambda_3=\frac{mqw^*}{(m_1+m_2w^*)^2}(m_2-\rho pm_1)$$
	So, the eigenvalues are negative whenever $\lambda_2<0$ which simplies to condition \ref{eq:(20)} and $\lambda_3<0$ which is ensured by condition \ref{eq:(17a)}. Hence, proved.
\end{proof}
\end{lemma}
Therefore, it is concluded that the state $E_{001}$ is unstable if it exists under \ref{eq:(17b)} or violates \ref{eq:(20)}.

\begin{theorem}{}\label{thm:3.7}
	The equilibirum state $E_1$ for the system \ref{eq:(2)} will be stable whenever one of the conditions is satisfied.
	
	\renewcommand{\theenumi}{(\roman{enumi})}%
	\begin{enumerate}
		\item The state $E_{01}$ of system \ref{eq:(16)} is stable.
		\item The state $E_{001}$ of system \ref{eq:(16)} exists and is stable.
		
	\end{enumerate}
\end{theorem}
\begin{proof}
	It is observed that
	\begin{itemize}
		\item $(x,y,w) \rightarrow E_{01}$ if and only if $(x,y,E) \rightarrow E_{1}$ when $y(t)\rightarrow 0$ faster than $E(t).$
		\item $(x,y,w) \rightarrow E_{001}$ if and only if $(x,y,E) \rightarrow E_{1}$ when $y(t)\rightarrow 0$ at a finite rate as $E(t)\rightarrow 0.$
	\end{itemize}
\end{proof}
\textbf{\underline{Note:}}
Hence, under conditions of Theorem \ref{thm:3.7} the system \ref{eq:(2)} with positive initial conditions will tend to a situation where only preys are left with the extinction of predators.\\	
It is noted from the existence and stability condition of $E_{01}$ and $E_{001}$ that both the points may exist simultaneously but only one can be stable at a time.\\
\subsection{Stability of }
The Jacobian matrix $J=[c_{ij}]_{3\times 3}$ for the original system \ref{eq:(2)} at a point $(x,y,E)$ is\\ $c_{11}=1-2x-\frac{ay^2}{(x+y)^2},\hspace{0.2cm}c_{12}=-\frac{ax^2}{(x+y)^2},\hspace{0.2cm}c_{13}=0$\\
$c_{21}=\frac{ey^2}{(x+y)^2},\hspace{0.2cm}c_{22}=-d+\frac{ex^2}{(x+y)^2}-\frac{qmm_1E^2}{(m_1E+m_2y)^2},\hspace{0.2cm}c_{23}=-\frac{qmm_2y^2}{(m_1E+m_2y)^2}$\\
$c_{31}=0,\hspace{0.2cm}c_{32}=\frac{pqmm_1E^2}{(m_1E+m_2y)^2},\hspace{0.2cm}c_{33}=\frac{pqmm_2y^2}{(m1E+m_2y)^2}-c$\\
At $E_2$,
$$J|_{E_2}=\begin{bmatrix}
	\frac{ax_1y_1}{(x_1+y_1)^2}-x&&-\frac{ax_1^2}{(x_1+y_1)^2}&&0\\
	\frac{ey_1^2}{(x_1+y_1)^2}&&-\frac{ex_1y_1}{(x_1+y_1)^2}&&-\frac{qm}{m_2}\\	0&&0&&\rho(\frac{pqm}{m_2}-c)\\
	
\end{bmatrix}$$
One of the eigenvalues of the above matrix is $\lambda_1=\rho(\frac{pqm}{m_2}-c) $.The other two eigenvalues are the roots of the equation
$$\mu_1^2+s_1\mu_1+s_2=0$$
where $s_1=x_1-\frac{(a-e)x_1y_1}{(x_1+y_1)^2}$,
$s_2=\frac{ex_1^2y_1}{(x_1+y_1)^2}.$\\

\vspace{0.5cm}
\begin{theorem}
	The effort free equilibrium $E_2$ is stable whenever
	\begin{equation}\label{eq:(21)}
		pqm-cm_2<0\hspace{0.2cm} \mbox{and}\hspace{0.2cm} a<\frac{e}{e+d}\bigg[\frac{e}{e-d}-d\bigg].
	\end{equation}
\end{theorem}
\begin{proof}
	
	The eigenvalue $\lambda_1$ of the Jacobian matrix at $E_2$ is negative if the first inequality in condition \ref{eq:(21)} is satisfied. The remaining two eigenvalues have positive product. So the eigenvalues will be negative only when the sum of eigenvalues is negative which is possible only when second inequality of condition \ref{eq:(21)}is satisfied.
	
\end{proof}
It is concluded that $E_3$ exists only if $E_2$ is unstable.
\subsection{Stability of interior}
The matrix $J$ computed at $E_3$ is
$$J|_{E_3}=
\begin{bmatrix}
	\frac{ax^*y^*}{(x^*+y^*)^2}-x^*&&-\frac{ax^*{^2}}{(x^*+y^*)^2}&&0\\
	\frac{ey^*{^2}}{(x^*+y^*)^2}&&-\frac{ex^*y^*}{(x^*+y^*)^2}+\frac{qmm_2y^*E^*}{(m_1E^*+m_2y^*)^2}&&-\frac{qmm_2y^*{^2}}{(m_1E^*+m_2y^*)^2}\\
	0&&\frac{\rho pqmm_1E^*{^2}}{(m_1E^*+m_2y^*)^2}&&-\frac{\rho pqmm_1E^*y^*}{(m_1E^*+m_2y^*)^2}
\end{bmatrix}=
\begin{bmatrix}
	a_{11}&&a_{12}&&0\\
	a_{21}&&a_{22}&&a_{23}\\
	0&&a_{32}&&a_{33}
\end{bmatrix}$$
The corresponding characteristic polynomial is 
\begin{equation}\label{eq:(22)}
	\mu_2^3+s_{11}\mu_2^2+s_{12}\mu_2+s_{13}=0
\end{equation}
where $s_{11}=-(a_{11}+a_{22}+a_{33})=-\bigg(\frac{(a-e)x^*y^*}{(x^*+y^*)^2}+\frac{qmy^*E^*(m_2-\rho pm_1)}{(m_1E^*+m_2y^*)^2}-x^*\bigg)$\\
$s_{12}=a_{11}a_{22}+a_{11}a_{33}+a_{22}a_{33}-a_{32}a_{23}-a_{12}a_{21}=\bigg(\frac{ex^{*2}y^*}{(x^*+y^*)^2}+\frac{qmx^*y^{*2}E^*(am_2-(a-e) pm_1\rho)}{(x^*+y^*)^2(m_1E^*+m_2y^*)^2}\bigg)$\\ 
and $s_{13}=-a_{11}a_{22}a_{33}+a_{11}a_{32}a_{23}+a_{33}a_{12}a_{21}=\frac{\rho pqmm_1ex^{2*}y^{2*}E}{(x^*+y^*)^2}$\\

\begin{theorem}{\label{thm:3.9}}
	The condition for stability of coexistence equilibrium point $E_3$ is given as 
	\begin{subequations}
		\begin{align}
			a<e \hspace{0.4cm} and \hspace{0.4cm} m_2<\rho p m_1.\label{eq:(23a)}\\
			and\hspace{0.5cm} s_{11}s_{12}-s_{13} >0.\label{eq:(23b)}
		\end{align}
	\end{subequations}
	\begin{proof}
		Using Routh-Hurwitz criterion for stability of $E_3$ the following conditions are required:
		$$
		s_{11}>0,\hspace{0.1cm} s_{12}>0,\hspace{0.1cm} s_{11}s_{12}-s_{13}>0
		$$
		Simplifying these, gives conditions \ref{eq:(23a)} and \ref{eq:(23b)}. 
	\end{proof}
\end{theorem}

\section{Bifurcation Analysis:}
Let us introduce the following new variables:\\
  $ a_1=\frac{aex^*y^*}{(x^*+y^*)^2}\hspace{0.2cm}\mbox{and}\hspace{0.2cm} b_1=\frac{\rho pqm_1m_2y^*E^*}{(m_1*E^*+m_2*y^*)^2}.$\\
The characteristic equation \ref{eq:(22)} is simplified as
\begin{equation}\label{eq:(24)}
	\mu_2^3+s_{11}\mu_2^2+s_{12}\mu_2+s_{13}=0
\end{equation}
where $s_{11}=x^*+\frac{a_1(e-a)}{ae}+\frac{mb_1(\rho pm_1-m_2)}{\rho pm_1m_2}$\\
$s_{12}=\frac{a_1x^*}{a}+\frac{ma_1b_1(e-a)}{eam_2}+\frac{mb_1x^*(\rho pm_1-m_2)}{\rho pm_1m_2}+\frac{ma_1b_1}{\rho pm_1e}.$\\ 
$s_{13}=\frac{ma_1b_1x^*}{am_2}>0$.\\
Now, $\Delta(m)=s_{11}s_{12}-s_{13}=C_1m^2+C_2m+C_3$\\ 
with $C_1=-\frac{a_1b_1^2(\rho pm_1-m_2)^2}{\rho^2 m_1^2m_2^2p^2e}+\frac{b_1^2x^*(\rho pm_1-m_2)^2}{\rho^2 m_1^2m_2^2p^2}+\frac{a_1b_1^2(\rho pm_1-m_2)}{\rho pm_1m_2^2a}$\\
$C_2=\frac{2a_1b_1x^*(e-a)(\rho pm_1-m_2)}{\rho pm_1m_2ae}+\frac{b_1x^{*2}(\rho pm_1-m_2)}{\rho pm_1m_2}+\frac{a_1^2b_1(e-a)}{\rho pm_1ae^2}+\frac{a_1^2b_1(e-a)^2}{e^2a^2m_2}$\\
$C_3=\frac{a_1x^{*2}}{a}+\frac{a_1^2x^*(e-a)}{ea^2}.$\\
\begin{theorem}
	If $E_3=(x^*,y^*,E^*)$ exists and condition \ref{eq:(23a)} is satisfied alongwith $C_1<0$, then hopf bifurcation exists in the neighbourhood of $E_3$ for $m^*=\frac{-C_2- \sqrt{C_2^2-4C_1C_3}}{2C_1}$.
\end{theorem}
\begin{proof}
	As $s_{11},s_{13}$ and $\Delta$ are smooth functions of $m$ in an open interval about $m^* \in \mathbb{R}$ such that 
	\renewcommand{\labelenumi}{(\roman{enumi})}
	\begin{enumerate}
		\item  The condition \ref{eq:(23a)} guarantees that $s_{11}>0$ .
		\item The condition \ref{eq:(23a)} gives, $C_2$ and $C_3>0$. The condition $C_1<0$ guarantees the positive root $m^*$ of $\Delta m$.
		\item ($\frac{d\Delta}{dm})|_{m=m^*}\neq 0$
		
	\end{enumerate}
	Hence by applying Liu's Criterion, hopf bifurcation occurs about $m=m^*.$
\end{proof}
\section{Global Stability:}
\begin{theorem}
	The interior equilibrium point is globally asymptotically stable in the domain of attraction given by
	\begin{equation}\label{eq:(25)}
		\frac{ay^*}{(x^*+y^*)}<x+y,\hspace{0.5cm} \frac{(m_1E+m_2y)}{(x+y)}<\frac{qmm_2E^*(x^*+y^*)}{ex^*(m_1E^*+m_2y^*)} 
	\end{equation}
\end{theorem}
\begin{proof}
	To analyse the global stability, construct the following positive definite function for arbitrarily chosen positive constants $l_i, i=1,2,3:$
	$$L(x,y,E)=l_1[(x-x^*)-x^*log(\frac{x}{x^*})]+l_2[(y-y^*)-y^*log(\frac{y}{y^*})]+l_3[(E-E^*)-E^*log(\frac{E}{E^*})]$$
	(It can be easily shown that $V(x^*,y^*,E^*)=0$ and positive for all positive values of $x,y,E.$)\\
	Now the time deriative of $V$ is
	$$\frac{dV}{dt}=l_1\frac{(x-x^*)}{x}\frac{dx}{dt}+l_2\frac{(y-y^*)}{y}\frac{dy}{dt}+l_3\frac{(E-E^*)}{E}\frac{dE}{dt}$$
	$$=l_1(x-x^*)(1-x-\frac{ay}{x+y})+l_2(y-y^*)(-d+\frac{ex}{x+y}-\frac{qEm}{(m_1E+m_2y)^2})+l_3(E-E^*)[\lambda(\frac{pqmy}{m_1E+m_2y}-c)]$$
	$$=-l_1(x-x^*)^2[1-\frac{l_1ay^*}{(x+y)(x^*+y^*)}]-l_2(y-y^*)^2[\frac{ex^*}{(x+y)(x^*+y^*)}-\frac{qmm_2E^*}{(m_1E+m_2y)(m_1E^*+m_2y^*)}]$$
	$$-\frac{l_3\lambda pqmm_1y^*(E-E^*)^2}{(m_1E+m_2y)(m_1E^*+m_2y^*)}+\frac{(x-x^*)(y-y^*)}{(x+y)(x^*+y^*)}[-l_1ax^*+l_2ey^*]$$
	$$+\frac{(E-E^*)(y-y^*)}{(m_1E+m_2y)(m_1E^*+m_2y^*)}[-l_2qm_2y^*+l_3\lambda pqmm_1E^*]$$ 
	Now letting $l_1=1,l_2=\frac{ax^*}{ey^*},l_3=\frac{am_2x^*}{\lambda epm_1E^*}$, we get
	$$\frac{dV}{dt}=-(x-x^*)^2[1-\frac{ay^*}{(x+y)(x^*+y^*)}]-(y-y^*)^2(\frac{ax^*}{ey^*})[\frac{ex^*}{(x+y)(x^*+y^*)}-\frac{qmm_2E^*}{(m_1E+m_2y)(m_1E^*+m_2y^*)}]$$
	$$-\frac{am_2x^*}{\lambda epm_1E^*}\frac{\lambda pqmm_1y^*(E-E^*)^2}{(m_1E+m_2y)(m_1E^*+m_2y^*)}$$
	So, whenever condition \ref{eq:(25)} is satisfied,
	$$\frac{dV}{dt}<0.$$
	Accordingly, V is a Lyapunov function in the domain \ref{eq:(25)}.
\end{proof}
\section{Optimal Harvesting Policy:}
To arrive at an optimal harvesting policy, consider the following functional for maximization:  
\begin{equation}\label{eq:(26)}
	J=\int_{0}^{\infty} e^{-\delta t}(\frac{pqmy}{m_1E+m_2y}-c)E
\end{equation}
where $\delta$ represents the annual discount rate. \\

The aim is to optimize equation \ref{eq:(26)} with state constraints \ref{eq:(2)} using the Pontraygin's Maximal Principle\cite{clark1976mathematical}.\\
\\\\
Let $\lambda_i, i=1,2,3$ be the adjoint variables and $m$ is the control variable with $0\leq m \leq m_{max}$. The Hamiltonian function for the control problem is considered as
$$H=(\frac{pqmy}{m_1E+m_2y}-c)E+\lambda_1(x(1-x)-\frac{axy}{x+y})+\lambda_2 (\frac{exy}{x+y}-dy-\frac{qmEy}{m_1E+m_2y})$$
$$+\lambda_3 \lambda(\frac{pqmyE}{m_1E+m_2y}-cE)$$
Assuming that the control constraint $m$ is not binding i.e optimal solution does not occur at $0$ or $m_{max},$ then the singular control  is 
\begin{equation}\label{eq:(27)}
	\frac{\partial H}{\partial m}=\frac{pqyE}{m_1E+m_2y}+\lambda_2(-\frac{qyE}{m_1E+m_2y})+\lambda_3\frac{\lambda pqyE}{m_1E+m_2 y}=0
\end{equation}

The adjoint variables are evaluated using the equations
$$
\frac{d\lambda_1}{dt}=\delta \lambda_1 -\frac{\partial H}{\partial x}
$$
\begin{equation}\label{eq:(28)}
	\frac{d\lambda_1}{dt}=\delta \lambda_1-(\lambda_1(1-2x-\frac{ay^2}{(x+y)^2})+\lambda_2(\frac{ey^2}{(x+y)^2}))
\end{equation}
$$
\frac{d\lambda_2}{dt}=\delta \lambda_2 -\frac{\partial H}{\partial y}
$$
\begin{equation}\label{eq:(29)}
	\frac{d\lambda_2}{dt}=\delta \lambda_2-(\frac{pqmm_2y^2}{(m_1E+m_2y)^2}+\lambda_1(-\frac{ax^2}{(x+y)^2})+\lambda_2(-d+\frac{ex^2}{(x+y)^2}-\frac{qmm_1E^2}{(m_1E+m_2y)^2})$$
	$$\hspace{2cm}+\lambda_3(\frac{\lambda pqmm_1E^2}{(m_1E+m_2y)^2}))
\end{equation}
$$\frac{d\lambda_3}{dt}=\delta \lambda_3-\frac{\partial H}{\partial E}$$
\begin{equation}\label{eq:(30)}
	\frac{d\lambda_3}{dt}=\delta \lambda_3-[(\frac{pqmm_2y^2}{(m_1E+m_2y)^2}-c)+\lambda_2(-\frac{qmm_2y^2}{(m_1E+m_2y)^2})+\lambda_3(-\frac{\lambda pqmm_1Ey}{(m_1E+m_2y)^2})]
\end{equation}
Simplifying the above equations gives
$$\frac{d \lambda_1}{dt}=a_1 \lambda_1 +a_2 \lambda_2$$
$$\frac{d \lambda_2}{dt}=b_1+b_2 \lambda_1+b_3 \lambda_2+b_4 \lambda_3 $$
$$\frac{d \lambda_3}{dt}=c_1 +c_2 \lambda_2 +c_3 \lambda_3$$
The constants $a_i,b_i$ and $c_i$ in above expression are defined below\\ $a_1=\delta +x-\frac{axy}{(x+y)^2},$ $a_2=-\frac{ey^2}{(x+y)^2},$ $b_1=-\frac{pqmm_2y^2}{(m_1E+m_2y)^2},$ $b_2=\frac{ax^2}{(x+y)^2},$\\ $b_3=\delta+\frac{exy}{(x+y)^2}-\frac{qmm_2yE}{(m_1E+m_2y)^2},$ $b_4=-\frac{\lambda pqmm_1E^2}{(m_1E+m_2y)^2}$\\
$c_1=-\frac{pqmm_2y^2}{(m_1E+m_2y)^2}+c,$ $c_2=\frac{qmm_2y^2}{(m_1E+m_2y)^2},$ $c_3=\delta +\frac{\lambda pqmm_2Ey}{(m_1E+m_2y)^2e}$\\
After solving the values of $\lambda_i $ are obtained as
$$\lambda_1=-\frac{a_2}{a_1}\lambda_2$$
$$\lambda_2=\frac{a_1(-b_1c_3+b_4c_1)}{a_1b_3c_3-b_2a_2c_3-b_4c_2a_1}$$
$$\lambda_3=-\frac{(c_1+c_2)}{c_3}\lambda_2$$
Using the values of $\lambda_i, i=1,2,3$ in equation \ref{eq:(27)}, a value of $m$ which is the optimal value for harvesting of predators is obtained.

\section{Numerical Simulation:}
In this section, the dynamical behavior of the system \ref{eq:(2)} is analyzed numerically with different parametric values.

\begin{example}
	\normalfont
	For the parameters chosen as in Figure 1, it is observed that the equilibrium points $E_2$ and $E_3$ do not exist.  Also, the state $E_0$ is stable as the stability condition (i) of Theorem \ref{thm:3.4} is satisfied. Accordingly, all the solution trajectories starting with different initial conditions are converging to $E_0$ in Figure \ref{fig:Figure 1(a)}. This confirms stability of $E_0$. Since both the conditions of Theorem \ref{thm:3.7} are not satisfied, the equilibrium $E_1$ is not stable. The observation in Figure \ref{fig:Figure 1(b)} is in agreement with the result.
\end{example}
\begin{figure}[!ht]
\centering
	\subcaptionbox{\label{fig:Figure 1(a)}}{\includegraphics[width=2in,height=2in]{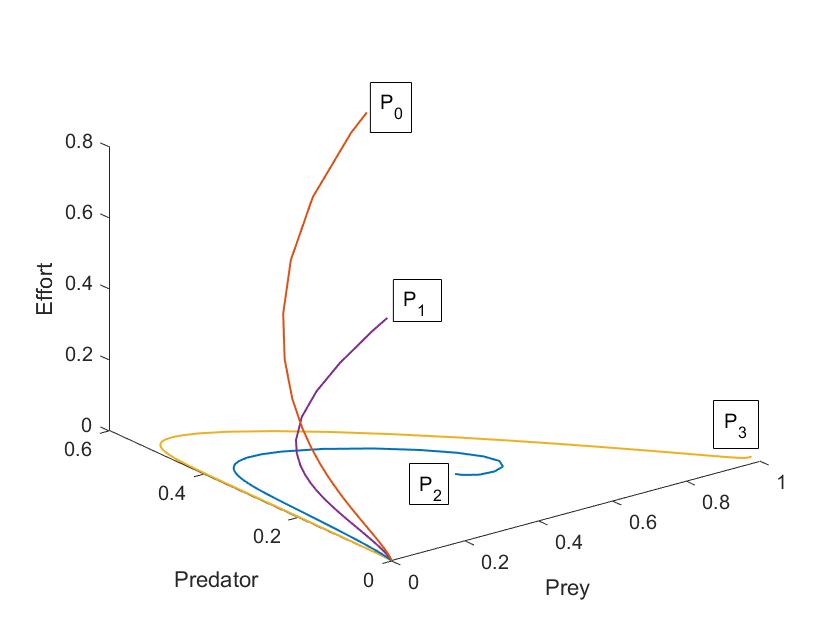}}
	\subcaptionbox{\label{fig:Figure 1(b)}}{\includegraphics[width=2in,height=2in]{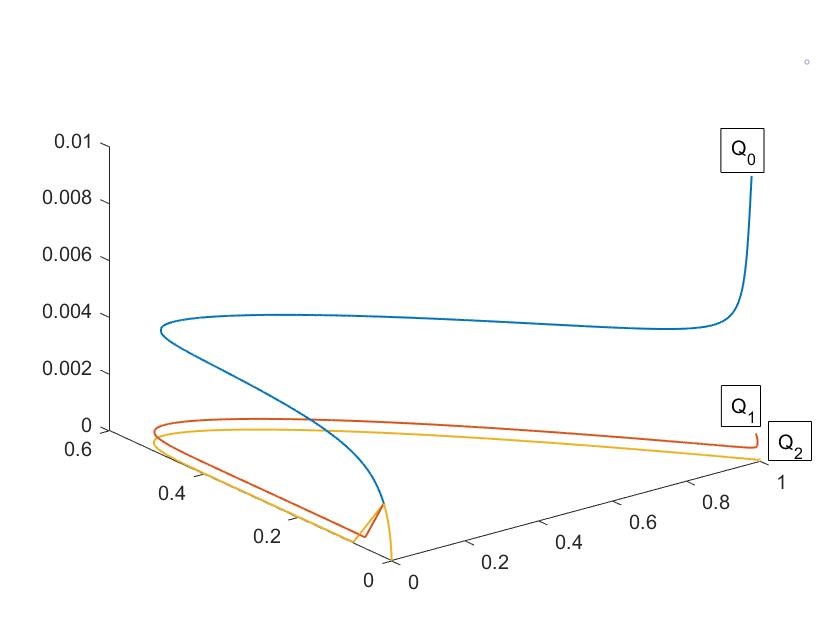}}
	
	{\textit{	\footnotesize Parametric Choice: $ a=2, c=0.6, d=0.07, e=0.6, q=0.6, m_1=0.4, m_2=0.4, m=0.5, \rho =1,p=0.8$.\\ Initial conditions (a): $P_0=(0.7,0.6,0.7), P_1=(0.5,0.4,0.3), P_2=(0.3,0.1,0.1),P_3=(0.99,0.01,0.01)$.\\ Initial conditions(b) $Q_0=(0.99,0.01,0.01), Q_1=(0.99,0.001,0.001), Q_2=(0.999,0.0001,0.0001)$.}
		
		\caption {Stability of $E_0$}}
	\label{Figure 1}
\end{figure}

\begin{example}
	\normalfont
	Both the equilibrium states $E_2$ and $E_3$ do not exist for the data set of Figure 2  also.  The equilibrium state $E_1$ is shown as attractor for various initial conditions in Figure \ref{Figure 2(a)}. Since none of the conditions of Theorem \ref{thm:3.7} is satisfied, the trajectories with initial conditions in the neighbourhood of $E_0$ are also shown to be attracted by $E_1$ (Figure \ref{Figure 2(b)}). It is concluded that $E_0$ is unstable in this case. 
\end{example}	

\begin{figure}[!ht]
\centering
	\subcaptionbox{\label{Figure 2(a)}}{\includegraphics[width=2in,height=2in]{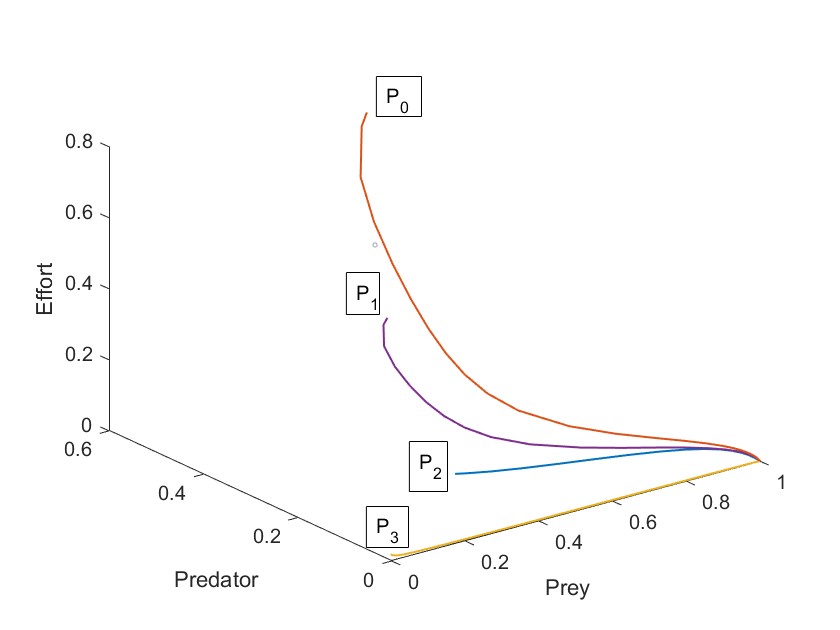}}
	\subcaptionbox{\label{Figure 2(b)}}{\includegraphics[width=2in,height=2in]{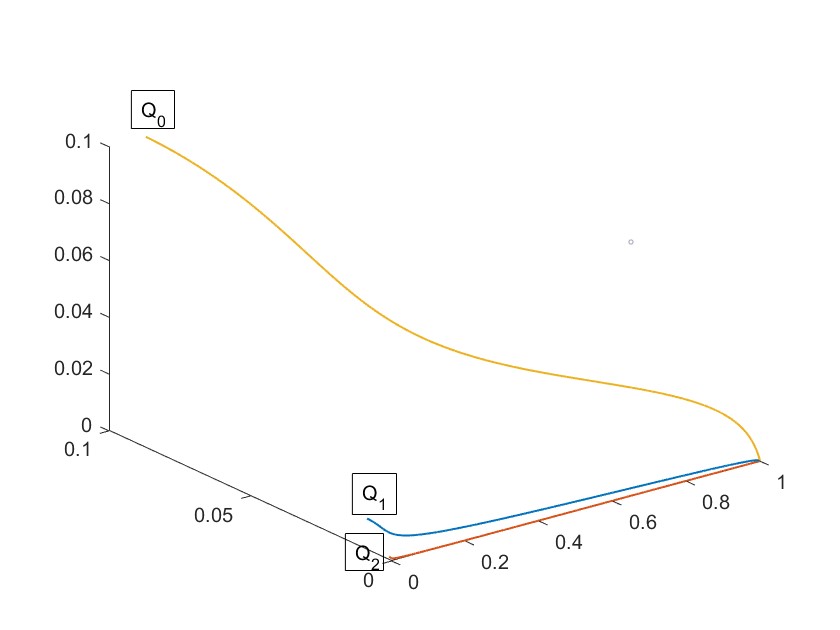}}
	\\	{\textit{	\footnotesize Parametric Choice:  $ a=2, 	c=0.6, d=0.07,
			e=0.6,
			q=0.8,
			m_1=0.4,
			m_2=0.4,
			m=0.6,
			\rho =1,p=1
			$. \\Initial conditions (a): $P_0=(0.7,0.6,0.7), P_1=(0.5,0.4,0.3), P_2=(0.3,0.1,0.1), P_3=(0.01,0.01,0.01).$\\ Initial conditions(b): $Q_0=(0.1,0.1,0.1), Q_1=(0.01,0.01,0.01), Q_2=(0.001,0.001,0.001)$. } 
		\caption{Stability of $E_1$}}
	\label{Figure 2}
\end{figure}
\begin{example}
	\normalfont
	In the subsequent examples, the parameters are suitably chosen so as to ensure the existence of $E_3$.
	With the parameteric values as in \ref{Figure 3}, the system \ref{eq:(2)} is found to be persistent for $ m>0.33$. This value is computed from \ref{lemma 2.3}. Figure \ref{Figure 3} confirms the persistence of the system \ref{eq:(2)}. In the Figure, \ref{Figure 3} it seems solution tending to zero but it is confirmed that they tend to non zero numerical values for very large $t$.
\end{example}

\begin{figure}[!ht]
\centering
	\subcaptionbox{\label{Figure 3(a)}}{\includegraphics[width=2in,height=2in]{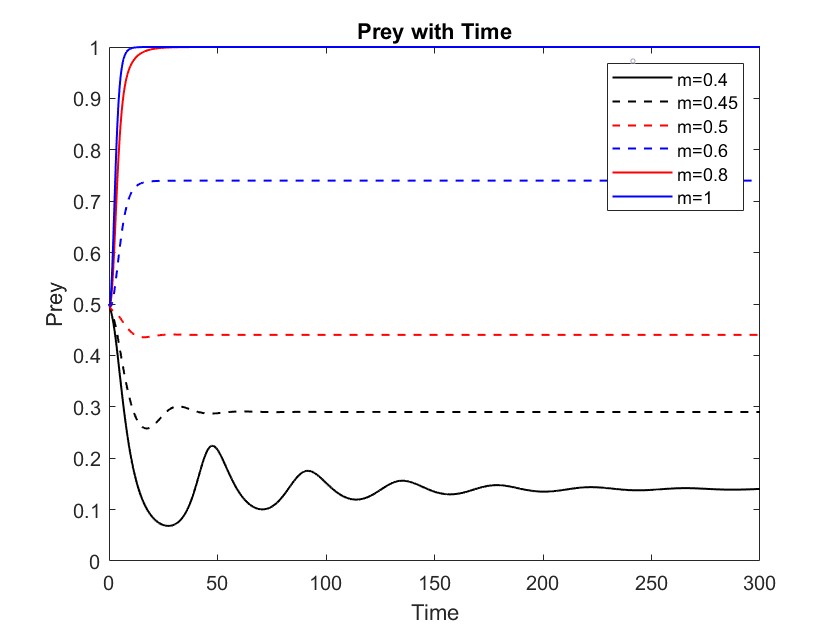}}
	\subcaptionbox{\label{Figure 3(b)}}{\includegraphics[width=2in,height=2in]{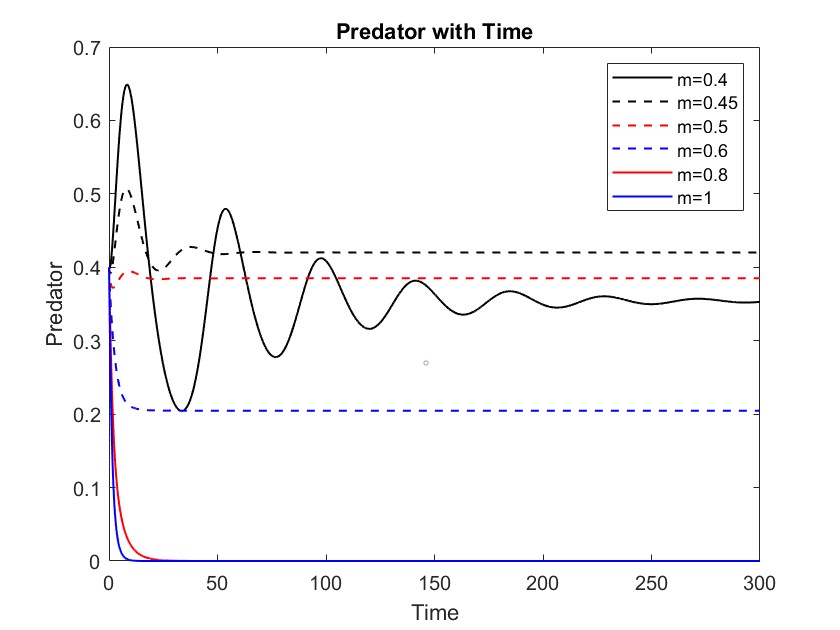}}

	{\subcaptionbox{\label{Figure 3(c)}}{\includegraphics[width=3in,height=2in]{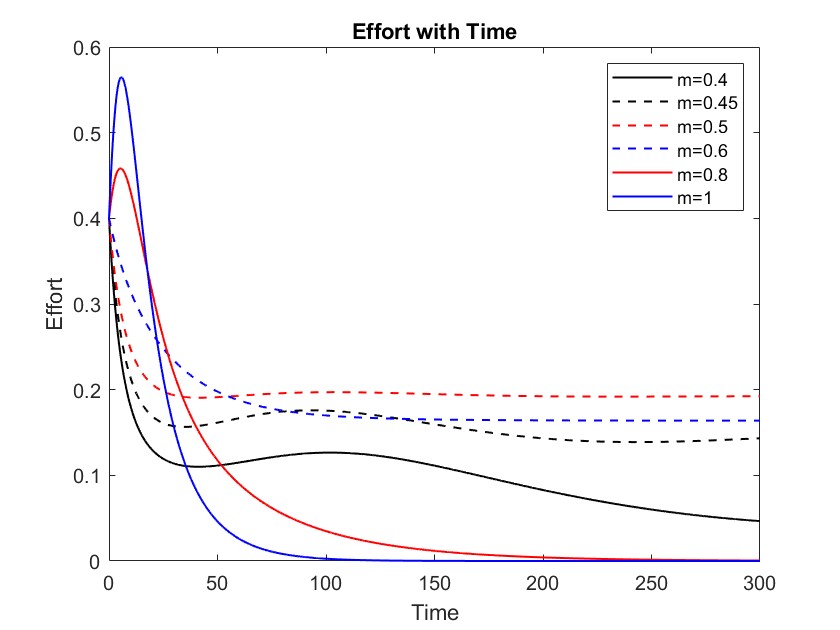}}}
\\	{\textit{\footnotesize  Parametric choice: $ a=1.2, c=3,
			d=0.07,
			e=0.6,
			q=0.6,
			m_1=0.4,
			m_2=0.4,
			p=6,
			\hspace{0.2cm} \mbox{and} \hspace{0.2cm} \rho =1. $ }}
	\caption{Time Series}
	\label{Figure 3}
\end{figure}  
It is observed that increasing $m$ will decrease the predator population as more predators are now available for harvesting. Consequently, there is an increase in prey population.
But the time series graph for effort is more interesting and versatile. It is seen with $m=0.35,m=0.4$ and $m=0.45$ the effort first decreases upto a critical level due to poor availability of predators  for harvesting. This causes a sharp rise in predator population. Once the predator grows to sizable numbers, the harvesting effort picks up and grows to the steady state level. During this time, the predator growth slowed down and reaches to the steady state level. For $m>0.8$, the high availability of predators leads to sharp increase in effort with fast decline in predator population.
\begin{example}
	\normalfont
	It can be noted from Theorem \ref{thm:3.9} for data choices of Figure \ref{fig:Figure 4} that only $E_3$ is stable. The solution  trajectories starting in the neighbourhood of $E_0$ and $E_1$ are attracted towards $E_3$. 
\begin{figure}[!ht]
		\centering
			\subcaptionbox{}{\includegraphics[width=3in,height=2in]{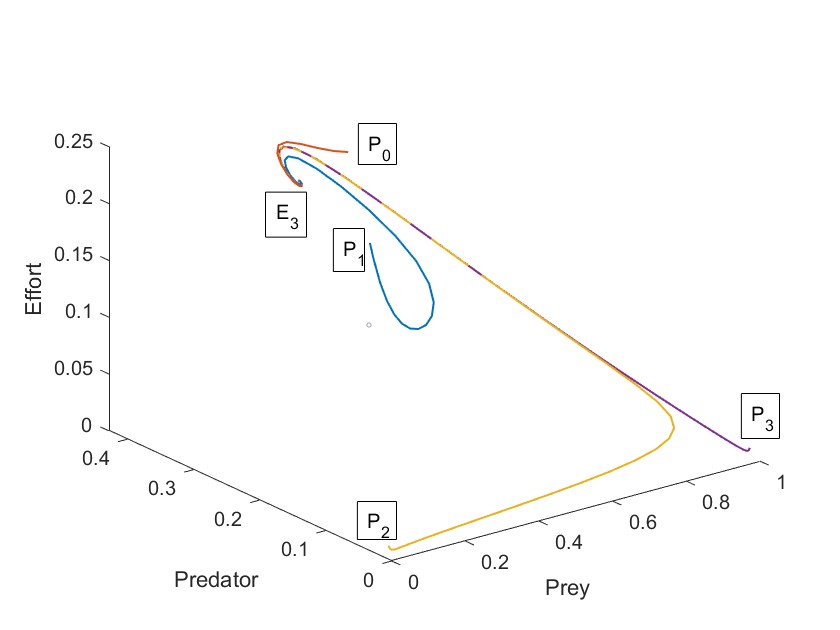}}\\
			{\textit{\footnotesize 	Parameteric Choice: $ a=1.2, c=3,
				d=0.07,
				e=0.6,
				q=0.6,
				m_1=0.4,
				m_2=0.4,
				m=0.5,
				p=6,
				\hspace{0.2cm} \mbox{and} \hspace{0.2cm} \rho =1. $\\ Initial conditions: $P_0=(0.3,0.2,0.2), P_1=(0.6,0.4,0.2), P_2=(0.01,0.01,0.01), P_3=(0.99,0.01,0.01)$ }}.
		\caption{Stability of coexistence point}
		\label{fig:Figure 4}
	\end{figure}
\end{example}

\begin{example}
	\normalfont
	Keeping all the parameters as in Figure \ref{fig:Figure 4} except $m=0.385$, the existence of Hopf bifurcation is expected. It is verified in Figure \ref{Figure 5}.
\end{example}

\begin{figure}[!ht]
\centering
	\subcaptionbox{\label{Figure 5(a)}}{\includegraphics[width=2in,height=2in]{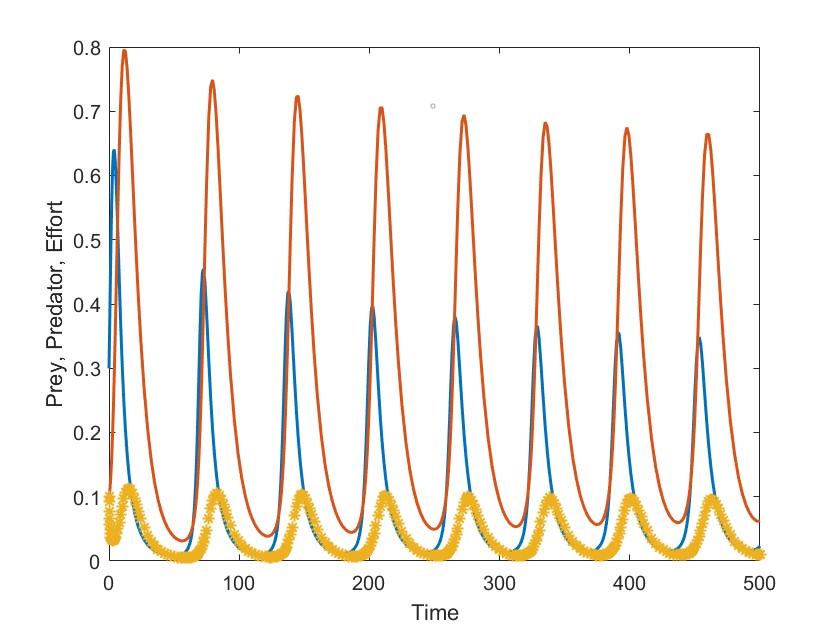}}
	\subcaptionbox{\label{Figure 5(b)}}{\includegraphics[width=2in,height=2in]{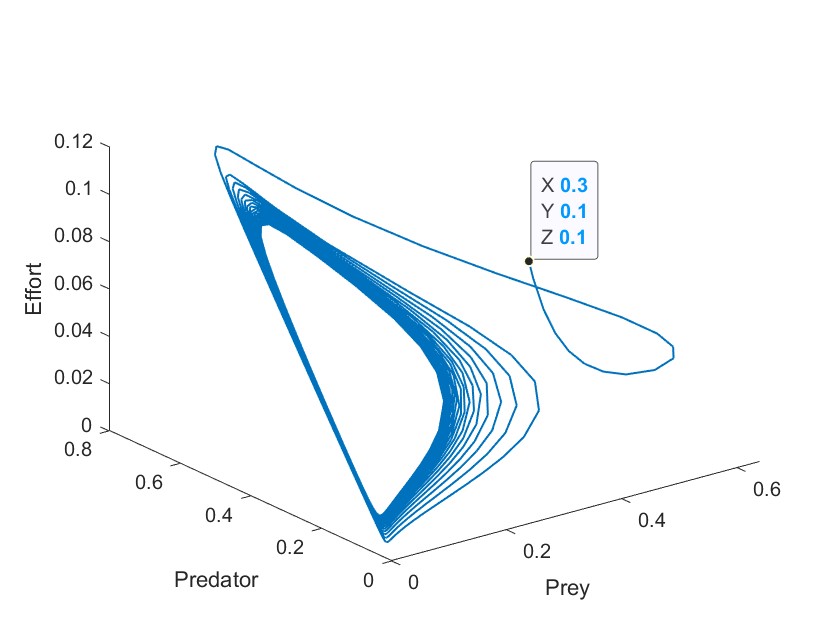}}
\\	{\textit{\footnotesize  Parametric choice: $ a=1.2, c=3, 
			d=0.07,
			e=0.6,
			q=0.6,
			m_1=0.4,
			m_2=0.4,
			m=0.385,
			p=6,
			\hspace{0.2cm} \mbox{and} \hspace{0.2cm} \rho =1. $ }}
	\caption{Hopf bifurcation}
	\label{Figure 5}
\end{figure}

Using Pontraygin's Maximal Principle explained in Section 6 and persistance condition from  Lemma \ref{lemma 2.3} we found the optimal value of $m$ as 0.686 and the coexistence equilibrium point as $(0.9987,0.0017,0.0018).$

\begin{table}[!ht]
	\centering
	\begin{tabular}{l  c  r}
	\hline
		$L_i:d=\rho c+k_i,$&$M_i:d=\frac{m_2}{pm_1}c+n_i,$&$R_i:d=r_i$\\
		\hline
		$k_1=e-\frac{mq}{m_1},$&$n_1=\frac{m_1e-mq}{pm_1},$&$r_1=e$\\
		\hline
		$k_2=\frac{-mq}{m1},$&$n_2=\frac{(a-1)(\rho pmq-m_2)-pmq}{pm_1},$&$r_2=\frac{e^2+\sqrt{e^2-4e^2(e+a)(1-a)}}{2(e+a)}$\\
		\hline $k_3,=e-\frac{\rho pmq}{m_2},$&&$r_3=d=\frac{e(a-1)}{a}$\\
		\hline
		$k_4=-\frac{\rho pmq}{m_2},$&$P_1:c=\frac{pmq}{m_2}$ &\\
		\hline
		\end{tabular}
	\caption{Equations of lines drawn in Figure 6.}
	\label{Table1}
\end{table}
\begin{table}[!ht]
\centering
	\begin{tabular}	{l c}
	\hline
		\textbf{Region}& \textbf{Existence and Stability point}\\
		\hline
		Region (I)&Stability of $E_1$\\
		
		Region (II)&  Stability region of $E_0$ only \\
	
		Region (III)& Stability region of $E_0$ and Existence region of $E_3$\\
		
		Region (IV) & Existence region of $E_3$ \\
	
		Region (V) & Stability region of $E_0$ and $E_2$\\
		
		Region (VI) & Stability region of $E_2$ only.\\
		\hline
	\end{tabular}
	\caption{Regions of Figure 6.}
\label{Table2}
\end{table}
\begin{figure}[!ht]
	\centering
	\includegraphics[width=7in,height=2in]{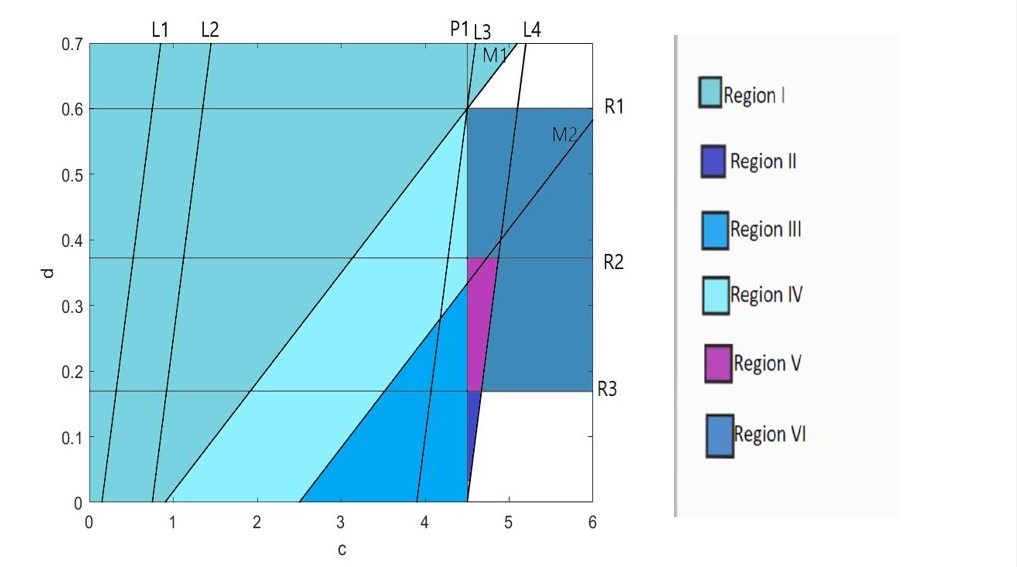}
	{\textit{\footnotesize Parameteric Choice:
			$a=1.4, e=0.6, q=0.6, m_1=0.4, m_2=0.4, p=6, \rho=1, m=0.5$
			on $c-d$ plane}}
		\caption{Regions of Existence and Stability}
		\label{Figure 6}
\end{figure}
The region of existence and stability of various equilibrium points  are examined in Figure \ref{Figure 6}. The bifurcation diagram is drawn with respect to the parameters $c$ and $d$. The equations of different lines in the Figure are given in Table \ref{Table1}. These lines divide the parameter plane into several regions. Considering various theorems and propositions discussed in the text, the behavior of equilibrium points (their existence and stability) in different regions are summarized in the Table \ref{Table2}.

It can be concluded from the Figure \ref{Figure 6} that there may be bistability of $E_0$ and $E_3$ in Region (III). The  parametric values in Figure \ref{Figure 7} are selected so as the data set lies in region III of the bifurcation diagram. This is the region of stability of $E_0$ and existence of $E_3$. It is observed that the solution trajectories with different initial conditions tend to different states. In Figure \ref{Figure7(a)} and \ref{Figure 7(b)} the initial conditions are very close, still they approach to different equilibrium states. In Figure \ref{Figure 7(c)} again the trajectories are going to different states. This confirms the existence of bi-stability in Region III. It may be noted that stability of $E_3$ could not be established analytically throughout in the Region III, however the data set satisfies the stability condition of Theorem \ref{thm:3.9} numerically.\\

\begin{figure}[!ht]
\centering
		\subcaptionbox{\label{Figure7(a)}}{\includegraphics[width=2in,height=2in]{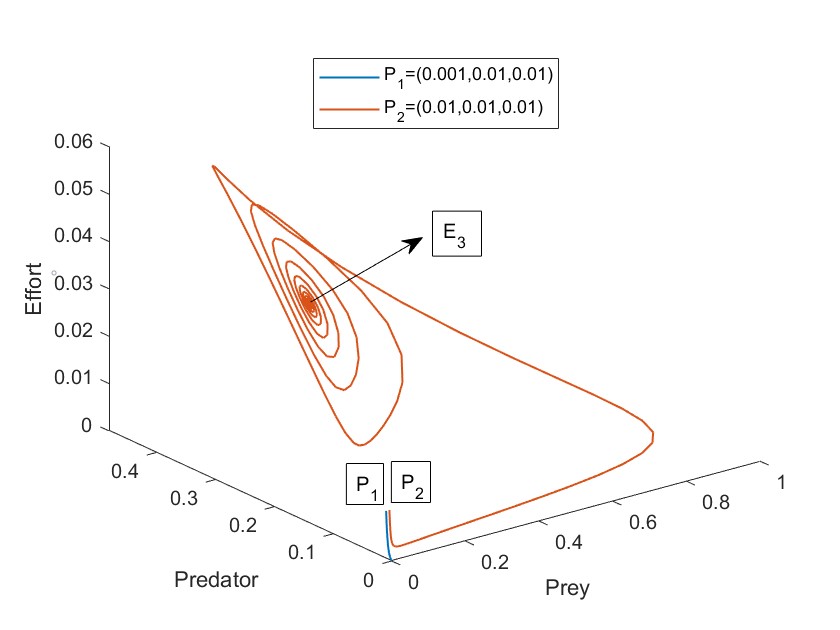}}
	\subcaptionbox{\label{Figure 7(b)}}{\includegraphics[width=2in,height=2in]{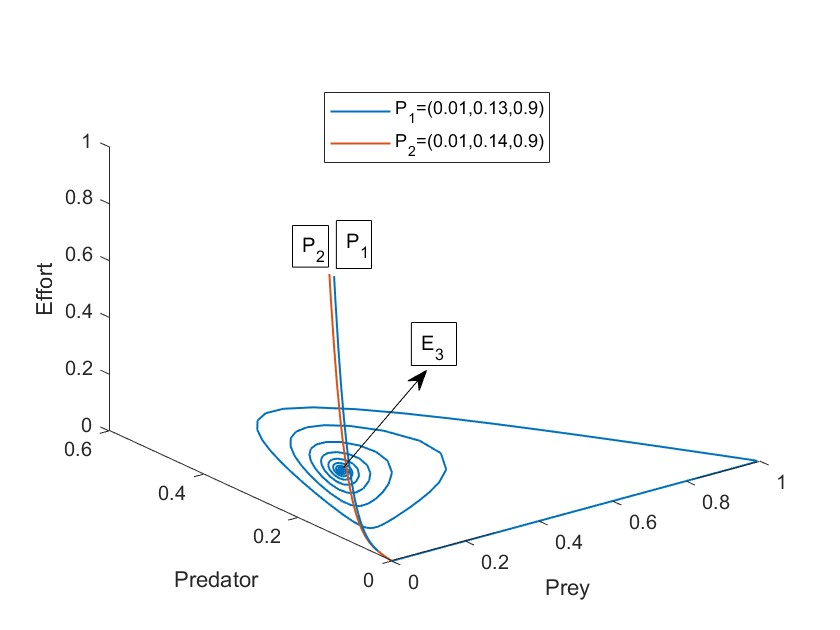}}
			\subcaptionbox{\label{Figure 7(c)}}{\includegraphics[width=3in,height=2in]{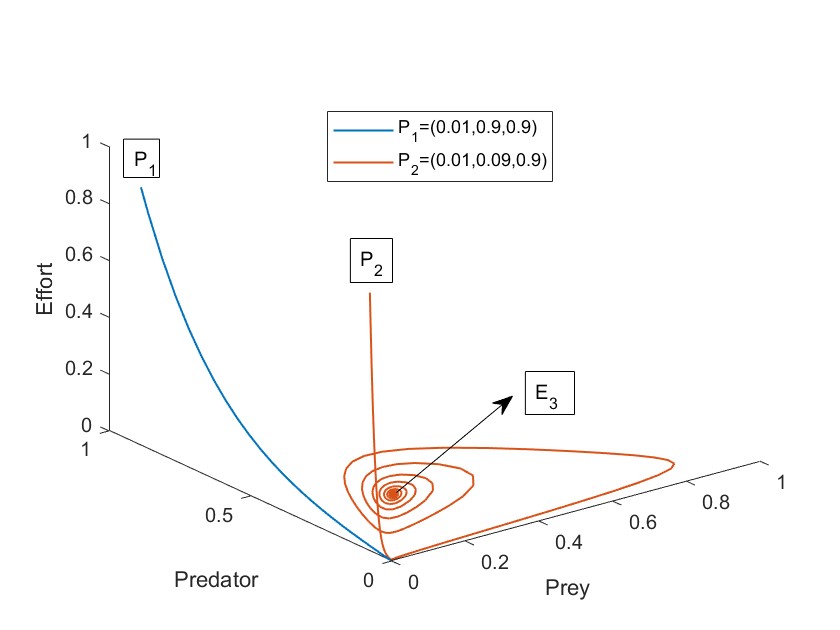}}
			
	{\textit{\footnotesize Parametric Choice: $a=1.4, c=4, d=0.18, e=0.6, q=0.6, m_1=0.4, m_2=0.4, p=6, \rho=1, m=0.5 $.}}
	\caption { Bistability of $E_0$ and $E_3$}
	\label{Figure 7}
\end{figure}
\newpage
\begin{figure}[!ht]
	\subcaptionbox{\label{Figure8(a)}}{\includegraphics[width=2in,height=2in]{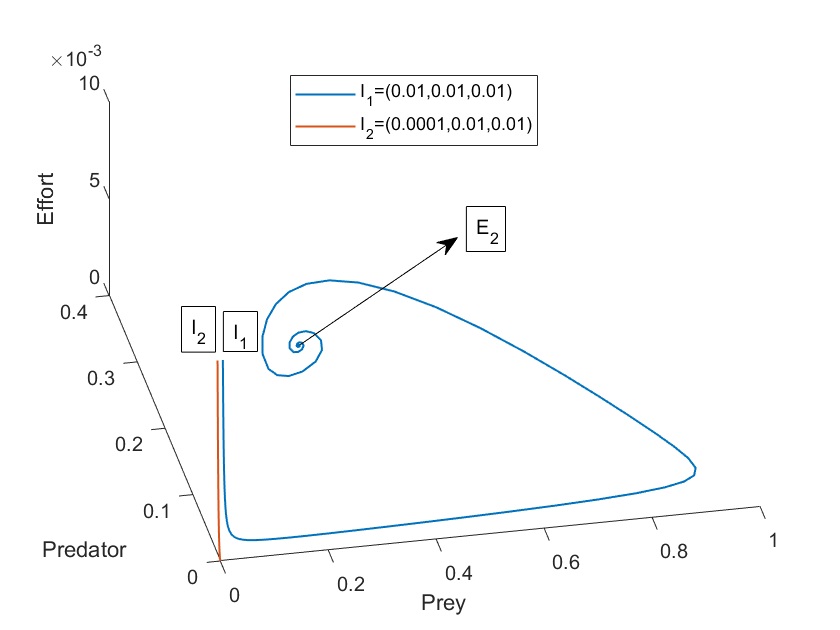}}
	\subcaptionbox{\label{Figure 8(b)}}{\includegraphics[width=2in,height=2in]{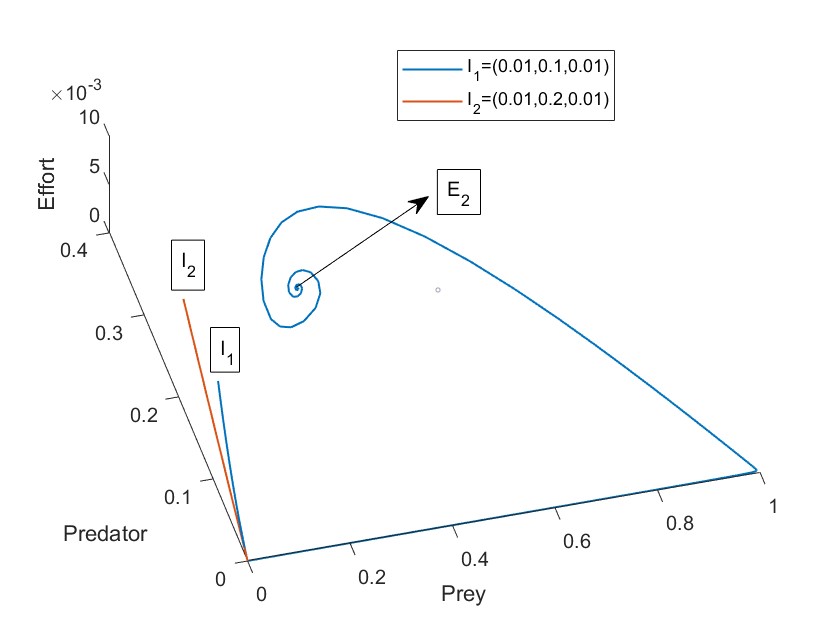}}

	{\textit{\footnotesize Paramteric Choice: $a=1.4, c=4.6, d=0.3, e=0.6, q=0.6, m_1=0.4, m_2=0.4, p=6, \rho=1, m=0.5 $.}}
		\caption {Bistability of $E_0$ and $E_2$}
	\label{Figure 8}
\end{figure}
Similarly, the region (V) of Figure \ref{Figure 6} is also a bistability region of $E_0$ and $E_2$. This is shown in Figure \ref{Figure 8}.
\newpage
\section{Conclusion:}
The effort dynamics of a ratio-dependent predator-prey system with non-linear harvesting is analyzed. The fundamental mathematical properties such as the existence and positivity of the solution are proved. The stability about the equilibrium states is examined using the eigenvalues of the Jacobian matrix. Although the system is not defined about origin, the behaviour of the system around origin is studied using blow-up technique. It is observed that though the initial conditions are positive still the system collapses. This collapse is of two kinds: firstly the whole system collapses i.e. eradiction of both prey and predator species and secondly extinction of the predator only. The  stability of the coexistence equilibrium point is analysed by applying Lyapunov method. Hopf bifurcation with respect to parameter $m$ is obtained in the neighbourhood of coexistence state. All the results that are proved analytically are also verified numerically with different parametric values. The bistability regions are identified and bi-stability is verified for choice of initial conditions. The optimal harvesting value of the fraction of predators available for harvesting is calculated in accordance with Pontraygin's Maximal Principle.

\section{Acknowlegement:}
The author (U. Yadav) is thankful to the "Ministry of Human Resource Development (MHRD)", Government of India, for providing financial support throughout this work (MHR-01-23-200-428).

\bibliographystyle{elsarticle-num}



\bibliography{mybibfile}
\end{document}